\documentclass{article}

\usepackage[utf8]{inputenc}

\usepackage{parskip, amsfonts, amsmath, changepage, amsthm, setspace, geometry, verbatim, MnSymbol,chemarrow}

\usepackage[pdftex]{graphicx}
\graphicspath{ {./images/} }

\usepackage[T1]{fontenc}

\usepackage{authblk}

\usepackage{xr}

\usepackage{MnSymbol, wasysym, mathtools}

\usepackage[tiny]{titlesec}

\usepackage{hyperref}

\newcommand*\tp[1]{\big( \begin{smallmatrix}#1\end{smallmatrix} \big)}

\newcommand*\ab[1]{\left\langle #1 \right\rangle}

\newcommand*\Mod[1]{ \; (\textup{mod} \; #1 )}

\newcommand*\tops[2]{\texorpdfstring{#1}{#2}}

\newcommand*\ol[1]{\overline{#1}}

\newcommand*{\Z}{\mathbb{Z}}
\newcommand*{\N}{\mathbb{N}}
\newcommand*{\R}{\mathbb{R}}
\newcommand*{\Q}{\mathbb{Q}}

\newcommand*\Cay{\textup{Cay}}

\renewcommand{\phi}{\varphi}

\newcommand*\Sec[1]{*{#1} \phantomsection \addcontentsline{toc}{section}{#1}}

\newtheorem{Thm}{Theorem}[section]

\theoremstyle{definition}

\theoremstyle{remark}

\newtheorem{Rmk}[Thm]{Remark}

\providecommand{\keywords}[1]
{
  \small	
  \textbf{\textit{Keywords---}} #1
}

\title{An alternate proof of Payan’s theorem on cubelike graphs}
\author[a]{Jonathan Cervantes}
\author[b]{Mike Krebs}
\affil[a]{University of California, Riverside, Dept. of Mathematics, Skye Hall, 900 University Ave., Riverside, CA 92521, jcerv092@ucr.edu}
\affil[b]{California State University, Los Angeles, Dept. of Mathematics, 5151 State University Drive, Los Angeles, CA 91711, mkrebs@calstatela.edu}
\date{\today}

\begin{document}

\maketitle


\begin{abstract}

A cubelike graph is a Cayley graph on the product $\mathbb{Z}_2\times\cdots\times\mathbb{Z}_2$ of the integers modulo $2$ with itself finitely many times.  In 1992, Payan proved that no cubelike graph can have chromatic number $3$.  The authors of the present paper previously developed a general matrix method for studying chromatic numbers of Cayley graphs on abelian groups.  In this note, we apply this method of Heuberger matrices to give an alternate proof of Payan's theorem.
\end{abstract}

\section{Introduction}


Given a finite set $A$, we take $\mathcal{P}(A)$ to be the power set of $A$.  We have that $\mathcal{P}(A)$ is an abelian group under the operation of symmetric difference, that is, $X\triangle Y=(X\setminus Y)\cup (Y\setminus X)$.  A {\it cubelike graph} is a Cayley graph whose underlying group is $\mathcal{P}(A)$.  Equivalently, writing $A=\{x_1,\dots,x_n\}$ and identifying the set $X\subset A$ with the $n$-tuple whose $i$th component is $1$ if $x_i\in X$ and is $0$ otherwise, a cubelike graph can be regarded as a Cayley graph whose underlying group is an $n$-fold product $\mathbb{Z}_2^n=\mathbb{Z}_2\times\cdots\times\mathbb{Z}_2$, where $\mathbb{Z}_2$ is the group of the integers modulo $2$ under addition.

Chromatic numbers of cubelike graphs have been studied by many authors.  One notable result is due to Payan \cite{Payan}, who proved that the chromatic number of a nonbipartite cubelike graph is always at least $4$.  That is, the chromatic number of a cubelike graph cannot equal $3$.  Publications with other results on chromatic numbers of cube-like graphs include \cite{Kokkala-et-al} and \cite[Section 9.7]{Jensen}.

Payan's proof is rather clever.  It is, however, somewhat {\it ad hoc}.  The purpose of the present note is to furnish an alternate proof of Payan's theorem, one that may lend itself naturally to generalizations.

Indeed, in \cite{Cervantes-Krebs-general}, the authors put forward a general method for approaching the problem of finding the chromatic number of a Cayley graph on an abelian group.  We show that Payan's theorem falls out quite naturally as a byproduct of this ``method of Heuberger matrices.''

The other key ingredient in our proof is a special case of Payan's theorem due to Sokolov\'{a} \cite{Sokolova}, who computed that even-dimensional cubes-with-diagonals (defined below) have chromatic number $4$.  The key idea of our proof of Payan's theorem is that if a cubelike graph is nonbipartite, then there is a graph homomorphism to it from an even-dimensional cube-with-diagonals.  The Heuberger matrices make transparent the existence of this homomorphism.

This note depends heavily on \cite{Cervantes-Krebs-general}, which we will refer to frequently.  The reader should assume that all notation, terminology, and theorems used but not explained here are explained there.

\section{Payan's theorem}\label{section-payans-theorem}

In this section we prove the following theorem.

\begin{Thm}[\cite{Payan}]\label{theorem-Payan}A cube-like graph cannot have chromatic number 3.\end{Thm}

Throughout this section we take cube-like graph to be a Cayley graph on $\mathbb{Z}_2^n$.

A special case of Theorem \ref{theorem-Payan} had previously been proven by Sokolov\'{a} in \cite{Sokolova}.  We will derive Payan's theorem from Sokolov\'{a}'s theorem, and for that reason we begin by discussing the latter.

For a positive integer $n$, the \emph{$n$-dimensional cube-with-diagonals graph} $Q^d_n$ is defined by \[Q_n^d=\text{Cay}(\mathbb{Z}_2^n,\{e_1,\dots,e_n,w_n\}),\] where $e_j$ is the $n$-tuple in $\mathbb{Z}_2^n$ with $1$ in the $j$th entry and $0$ everywhere else, and $w_n$ is the $n$-tuple in $\mathbb{Z}_2^n$ with $1$ in every entry.  We can visualize $Q^d_n$ as a hypercube with edges (called ``diagonals,'' hence the name and the superscript `$d$') added to join each pair of antipodal vertices.  Sokolov\'{a} proved that for $n$ even, $Q^d_n$ has chromatic number $4$.  We present here a condensed version of the proof in \cite{Sokolova} of this result.

\begin{Thm}[\cite{Sokolova}]\label{theorem-Sokolova}If $n$ is even, then $\chi(Q_n^d)=4$.\end{Thm}

\begin{proof}First observe that $(x_1,\dots,x_n)\mapsto (x_1,x_2+\cdots+x_n)$ defines a group homomorphism from $\mathbb{Z}^n_2$ to $\mathbb{Z}^2_2$ mapping $\{e_1,\dots,e_n,w_n\}$ to $\{(1,0),(0,1),(1,1)\}$.  So this defines a graph homomorphism from $Q^d_n$ to $Q_2^d\cong K_4$, the complete graph on $4$ vertices.  Hence $\chi(Q_n^d)\leq 4$.

Next we show that $Q^d_n$ is not properly $3$-colorable.  We do so by induction.  For the base case ($n=2$), we saw previously that $Q^d_2\cong K_4$, which is not properly $3$-colorable.  Now assume that $Q^d_n$ is not properly $3$-colorable, and we will show that $Q^d_{n+2}$ is not properly $3$-colorable.  Suppose to the contrary that $c\colon \mathbb{Z}^{n+2}_2\to\mathbb{Z}_3$ is a proper $3$-coloring.  For two tuples $v=(v_1,\dots,v_j)$ and $u=(u_1,\dots,u_k)$, we define $v*u=(v_1,\dots,v_j,u_1,\dots,u_k)$.  Define $c'\colon\mathbb{Z}^{n}_2\to\mathbb{Z}_3$ by $c'(v)=k$ if $\{c(v*(0,0)),c((v+w_n)*(1,0))\}$ equals either $\{k\}$ or $\{k,k+1\}$.  A straightforward case-by-case analysis shows that $c'$ is a proper $3$-coloring of $Q^d_n$, which is a contradiction.\end{proof}

\begin{Rmk}We briefly digress to remark that Sokolov\'{a}'s theorem can be restated as follows.  In any (not necessarily proper) $3$-coloring of the vertices of an even-dimensional hypercube, there must exist two antipodal vertices, both of which are assigned the same color.  Stated this way, it brings to mind various topological theorems such as the hairy ball theorem and the Borsuk–Ulam theorem.   
We wonder whether there might be some connection between Sokolov\'{a}'s combinatorial result and one or more of these facts from topology, perhaps along the lines of the connection between Sperner's lemma and the Brouwer fixed point theorem.\hfill$\square$\end{Rmk}

Using Heuberger matrices, we will now see how Sokolov\'{a}'s theorem implies Payan's theorem.  The key idea is to show that every nonbipartite cube-like graph contains a homomorphic image of an even-dimensional cube-with-diagonals graph.

\begin{proof}[Proof of Theorem \ref{theorem-Payan}]Let $X=\text{Cay}(\mathbb{Z}_2^n,S)$ be a nonbipartite cube-like graph.  Because $2x=0$ for all $x\in S$, there is a Heuberger matrix $M_X$ associated to $X$ whose last $m$ columns are $2e_1,\dots,2e_m$, where $m=|S|$.  That is, $M_X$ has the form $(A\;\vert\;2I_m)$ for some integer matrix $A$.  Here $I_m$ is the $m\times m$ identity matrix.  Using column operations as in \cite[Lemma \ref{general-lemma-isomorphisms}]{Cervantes-Krebs-general}, we have that \[(A\;\vert\;2I_m)_X^{\text{SACG}}\cong (A'\;\vert\;2I_m)_X^{\text{SACG}},\] for some matrix $A'$ whose entries are all in $\{0,1\}$.  Because $X$ is nonbipartite, by \cite[Lemma \ref{general-lemma-bipartite}]{Cervantes-Krebs-general}, some column $y$ of $A'$ contains an odd number $z$ of nonzero entries.  Hence by \cite[Lemma \ref{general-lemma-homomorphisms}, parts (\ref{general-lemma-homomorphisms-append-columns}) and (\ref{general-lemma-homomorphisms-append-zero-row})]{Cervantes-Krebs-general}, we have homomorphisms\[(w_z^t\;\vert\;2I_z)_Y^{\text{SACG}}\xrightarrow[\tau_1]{\ocirc}(y\;2e_{i_1}\;\cdots \;2e_{i_z})^{\text{SACG}}\xrightarrow[\tau_2]{\ocirc}(A'\;\vert\;2I_m)_X^{\text{SACG}}\]where $i_1,\dots,i_z$ are the indices of the nonzero entries of $y$, and $w_z^t$ is a column vector of length $z$ with a $1$ in every entry.  For $\tau_1$, we insert zero rows as appropriate; for $\tau_2$ we append the requisite columns.  So $\chi(Y)\leq\chi(X)$ by \cite[Lemma \ref{general-lemma-pullback}]{Cervantes-Krebs-general}.  If $z=1$, then $X$ has loops and is not properly colorable.  So assume $z\geq 3$.  Observe that $Y\cong Q_{z-1}^d$.  An application of Theorem \ref{theorem-Sokolova} then completes the proof.\end{proof}

\bibliographystyle{amsplain}
\bibliography{references}



\end{document}